\documentclass{article}
\usepackage{amssymb}
\usepackage{amsmath}
\usepackage{amsfonts}
\usepackage{amscd}
\usepackage{epsfig}
\usepackage{amsmath,amstext,amsthm,amsbsy,amssymb}

\newcommand{\bi}{{\bf i}}
\newcommand{\bj}{{\bf j}}
\newcommand{\bk}{{\bf k}}
\newcommand{\bc}{{\mathbb C}}
\newcommand{\br}{{\mathbb R}}
\newcommand{\bh}{{\mathbb H}}

\newtheorem{thm}{Theorem}[section]
\newtheorem{lem}{Lemma}[section]
\newtheorem{pro}{Proposition}[section]
\newtheorem{cor}{Corollary}[section]
\newtheorem{rem}{Remark}[section]
\newtheorem{defi}{Definition}[section]
\newtheorem{exam}{Example}[section]

\begin{document}

\title{The Moore-Penrose inverses of split quaternions}
\author{ Wensheng Cao,  Zhenhu Chang \\
School of Mathematics and Computational Science,\\
Wuyi University, Jiangmen, Guangdong 529020, P.R. China\\
e-mail: {\tt wenscao@aliyun.com}\\
}
\date{}
\maketitle

\bigskip
{\bf Abstract} \,\,  In this paper, we find the roots of lightlike quaternions.  By introducing the concept of the Moore-Penrose inverse  in split quaternions, we solve the linear equations $axb=d$, $xa=bx$ and $xa=b\bar{x}$.  Also we obtain necessary and sufficient conditions for two split quaternions to be similar or consimilar.

\vspace{1mm}\baselineskip 12pt


{\bf MR(2000) Subject Classifications:}\ \ {\rm 15A33; 11R52}

\section{Introduction}

Let $\br$ and $\bc$ be the field of  real numbers and  complex numbers, respectively.  The split quaternions are elements of a 4-dimensional associative algebra introduced by James Cockle  \cite{cock}  in 1849.
Split quaternions can be represented as
$$\bh_s=\{q=q_0+q_1\bi+q_2\bj+q_3\bk,q_i\in \br,i=0,1,2,3\},$$
where $1,\bi,\bj,\bk$ are basis of $\bh_s$ satisfying the equalities:
\begin{equation}\label{rule}
\bi^2=-\bj^2=-\bk^2=-1, \bi\bj=\bk=-\bj\bi,\bj\bk=-\bi=-\bk\bj, \bk\bi=\bj=-\bi\bk.
\end{equation}
  Let $\bar{q}=q_0-q_1\bi-q_2\bj-q_3\bk$ be the conjugate of $q$ and   $I_q=\bar{q}q=q\bar{q}=q_0^2+q_1^2-q_2^2-q_3^2.$
A split quaternion is spacelike, timelike or lightlike if $I_q<0, I_q > 0$ or $I_q = 0$, respectively.
 We define $\Re(q)=(q+\bar{q})/2=q_0$ to be the real part of $q$ and $\Im(q)=(q-\bar{q})/2=q_1\bi+q_2\bj+q_3\bk$ to be the imaginary part of $q$. It is easy to verify that $I_{pq}=I_pI_q.$  Let $q'=q_0-q_1\bi+q_2\bj+q_3\bk$ be the prime of $q$. Then $I_{q'}=I_q$.
For $q=q_0+q_1\bi+q_2\bj+q_3\bk$, we define \begin{equation}\label{kp}K(q)=\Im(q)^2=-q_1^2+q_2^2 +q_3^2.\end{equation}
Obviously $\bc=\br+\br\bi\subset \bh_s$ and $\bj z=\bar{z}\bj$ for $z\in \bc.$  A split quaternion can be written  as $$q=(q_0+q_1\bi)+(q_2+q_3\bi)\bj =z_1+z_2\bj,z_1,z_2\in \bc.$$
Denote $\overrightarrow{x}=(x_0,x_1,x_2,x_3)^T$  for
$x=x_0+x_1\bi+x_2\bj+x_3\bk\in \bh_s$, where $^T$ denotes the transpose of a  matrix.  Using the multiplication rule (\ref{rule}), we have the following formulas:
$$\overrightarrow{qx}=L_q \overrightarrow{x},\overrightarrow{xq}=R_q \overrightarrow{x},$$
 where $$
L(q)=\left(\begin{array}{cccc}
q_0 &-q_1  &q_2  &q_3  \\
q_1 &q_0   & q_3  & -q_2  \\
q_2 & q_3  &q_0   &-q_1  \\
q_3 &  -q_2 &q_1   & q_0
\end{array}\right),\ \  R(q)=\left(\begin{array}{cccc}
q_0 &-q_1  &q_2  &q_3  \\
q_1 &q_0   & -q_3  & q_2  \\
q_2 & -q_3  &q_0   &q_1  \\
q_3 &  q_2 &-q_1   & q_0
\end{array}\right).
$$

 Unlike the Hamilton quaternion algebra, the split quaternions contain nontrivial zero divisors, nilpotent elements, and idempotents. Accordingly we define the following sets of split quaternions:
 \begin{equation}
 Z(\bh_s)=\{q\in \bh_s:I_q=0\};
 \end{equation}
\begin{equation}
N(\bh_s)=\{q\in \bh_s:\mbox{there is an integer}\ n\  \mbox{such that}\ q^n=0\};
\end{equation}
 \begin{equation}
 I(\bh_s)=\{q\in \bh_s:q^2=q\}.
 \end{equation}

 In this paper we mainly concentrate on  the properties of lightlike  split quaternions or some singular matices induced by split quaternions.  In Section \ref{rootse}, we will obtain the characteristics of the sets $ Z(\bh_s),N(\bh_s), I(\bh_s)$ and  find the roots of lightlike quaternions.  In Section \ref{mpinvsec}, we will obtain some properties of  the Moore-Penrose inverse and solve the linear equation $axb=d$.    In Section \ref{simsec}, we solve the linear equations $xa=bx$ and $xa=b\bar{x}$. As applications, we  obtain some necessary and sufficient conditions for two split quaternions to be similar or consimilar.

\section{The root of a split quaternion in  $Z(\bh_s)$}\label{rootse}

\begin{pro}\label{zdivision}  $q^2=2\Re(q)q-I_q,  q\in \bh_s$ and   	
 \begin{equation}q^n=(2\Re(q))^{n-1}q, q\in Z(\bh_s). \end{equation}
\end{pro}

 \begin{proof} It is obvious that $q^2=q(2\Re(q)-\bar{q})=2\Re(q)q-I_q$.  If $q\in Z(\bh_s)$ then $I_q=0$, and therefore
 $q^n=(2\Re(q))^{n-1}q.$
 \end{proof}

 \begin{pro}\label{nhs} $N(\bh_s)=\{q=q_1\bi+q_2\bj+q_3\bk\  \mbox{with}\ q_1^2-q_2^2-q_3^2=0  \}.$
\end{pro}
\begin{proof} Let $q=q_0+q_1\bi+q_2\bj+q_3\bk\in N(\bh_s)$. Then $I_q=0$ and therefore $q\in Z(\bh_s)$. By proposition \ref{zdivision}, $\Re(q)=0$ which implies that $I_q=q_1^2-q_2^2-q_3^2=0$.
\end{proof}

  \begin{pro}$I(\bh_s)=\{0,1\}\cup \{\frac{1}{2}+q_1\bi+q_2\bj+q_3\bk\  \mbox{with}\ \frac{1}{4}+q_1^2-q_2^2-q_3^2=0  \}.$
 \end{pro}

 \begin{proof} Let $q=q_0+q_1\bi+q_2\bj+q_3\bk\in I(\bh_s)$. By $q^2=q$, we have  the following equations:
 	$$q_0^2-q_1^2+q_2^2+q_3^2=q_0,2q_0q_1=q_1,2q_0q_2=q_2,2q_0q_3=q_3.$$ If $q_1^2+q_2^2+q_3^2\neq 0$ then $q_0=\frac{1}{2}$ and   $\frac{1}{4}+q_1^2-q_2^2-q_3^2=0$.  If  $q_1^2+q_2^2+q_3^2=0$ then  $q$ is $0$ or $1$.
\end{proof}

Let $q=z_1+z_2\bj\in  Z(\bh_s)$. Then $I_q=|z_1|^2-|z_2|^2=0$, and therefore $q$ can be written as  $$q=r(e^{\bi\alpha}+e^{\bi\beta}\bj),r\ge 0,\alpha,\beta\in [0,2\pi).$$

  \begin{thm}\label{root}Let $q=r(e^{\bi\alpha}+e^{\bi\beta}\bj),r>0,\alpha,\beta\in [0,2\pi)$ and $n\ge 2$.  Then the equation $w^n=q$ has solution in the following cases:


  (1) $w=\sqrt[n]{\frac{r}{(2\cos\alpha)^{n-1}}}(e^{\bi\alpha}+e^{\bi\beta}\bj)$ if $\cos\alpha>0$; $\mbox{or}\ \cos\alpha<0,n \ \mbox{is  odd}$;

  (2) $w=\sqrt[n]{\frac{r}{(2\cos\alpha)^{n-1}}}(-e^{\bi\alpha}-e^{\bi\beta}\bj)$ if $\cos\alpha>0,n \ \mbox{is  even} $.
 \end{thm}

\begin{proof}
	Let $w$ be a root of the equation $w^n=q$. Then $(I_w)^n=I_q=0$ and therefore $w\in Z(\bh_s)$. So $w$ can be written as  $w=r_1(e^{\bi\alpha_1}+e^{\bi\beta_1}\bj),r_1>0,\alpha_1,\beta_1\in [0,2\pi)$.  It follows from Proposition \ref{zdivision} that
	\begin{equation}(2r_1\cos\alpha_1)^{n-1}r_1(e^{\bi\alpha_1}+e^{\bi\beta_1}\bj)=r(e^{\bi\alpha}+e^{\bi\beta}\bj).\end{equation}
That is
	$$(2r_1\cos\alpha_1)^{n-1}r_1\cos\alpha_1=r\cos\alpha,\ \ (2r_1\cos\alpha_1)^{n-1}r_1\sin\alpha_1=r\sin\alpha,$$	$$(2r_1\cos\alpha_1)^{n-1}r_1\cos\beta_1=r\cos\beta,\ \ (2r_1\cos\alpha_1)^{n-1}r_1\sin\beta_1=r\sin\beta.$$
It is obvious that there is no solution if $\cos\alpha=0$.  So we only need to consider the case $\cos\alpha\neq 0$. Thus $\tan \alpha_1=\tan \alpha.$   In this case, the above four equations are solvable if and only if one of the following conditions holds:
\begin{itemize}
  \item [(1)] $\alpha_1=\alpha,\beta_1=\beta,r_1=\sqrt[n]{\frac{r}{(2\cos\alpha)^{n-1}}} \ \mbox{provided}\ \cos\alpha>0;$

  \item [(2)]$\alpha_1=\alpha\pm \pi\in [0,2\pi),\beta_1=\beta\pm \pi\in [0,2\pi) ,r_1=\sqrt[n]{\frac{r}{(2\cos\alpha)^{n-1}}} \ \mbox{provided}\ \cos\alpha>0,n \ \mbox{is  even} ;$
        \item [(3)]$\alpha_1=\alpha,\beta_1=\beta,r_1=\sqrt[n]{\frac{r}{(2\cos\alpha)^{n-1}}} \ \mbox{provided}\ \cos\alpha<0,n \ \mbox{is  odd} .$
\end{itemize}
The above three cases conclude the proof of Theorem \ref{root}.
\end{proof}

\begin{rem}
	Proposition \ref{nhs}, Theorem \ref{root} and the results in \cite{moz} show how to find  the roots of any split quaternions.
\end{rem}

\section{The Moore-Penrose inverse of elements in $\bh_s$}\label{mpinvsec}

We recall that the Moore-Penrose inverse of a complex matrix $A$ is the unique  complex matrix $X$ satisfying the following equations:
$$AXA=A,XAX=X,(AX)^*=AX,(XA)^*=XA,$$
where $^*$ is the conjugate of a  matrix.  We denote the Moore-Penrose inverse of $A$ by $A^+$. The following lemma is well known.
\begin{lem}\label{general} Let $A\in \br^{m\times n},b\in
	\br^{m}$ and $E_n$ the identity matrix of order $n$. Then the linear equation $Ax=b$ has a solution if and only if $AA^{+}b=b$, furthermore the general solution is $$x=A^{+}b+(E_n-A^{+}A)y, \forall y\in \br^{n}.$$
\end{lem}
\begin{defi} The  Moore-Penrose inverse of  $a=c_1+c_2\bj\in
\bh_s$ is defined to be
$$a^+=\left\{\begin{array}{ll}
0, & \hbox{provided a=0;} \\
\frac{\bar{c_1}-c_2\bj}{|c_1|^2-|c_2|^2}=\frac{\bar{a}}{I_a}, & \hbox{provided $|c_1|^2-|c_2|^2\neq 0$;} \\
\frac{\bar{c_1}+c_2\bj}{4|c_1|^2}, & \hbox{provided $|c_1|^2-|c_2|^2=0$.} \\
\end{array}%
\right.$$
\end{defi}
For $0\neq a=c_1+c_2\bj\in  Z(\bh_s)$, we have the following facts:
 \begin{equation}\label{cinvf}aa^+a=a,a^+aa^+=a^+,aa^+=\frac{1}{2}(1+\frac{c_2}{\bar{c_1}}\bj),a^+a=\frac{1}{2}(1+\frac{c_2}{c_1}\bj).
 \end{equation}
Since $L(a)R(b)=R(b)L(a)$ for any $a,b\in \bh_s$, We can verify directly the following proposition.
\begin{pro}\label{mpprop} Let  $a=a_0+a_1\bi+a_2\bj+a_3\bk\in
	\bh_s$. Then
	$$L(a)L(a^+)L(a)=L(a), L(a^+)L(a)L(a^+)=L(a^+),L(a)L(a^+)=\big(L(a)L(a^+)\big)^*,L(a^+)L(a)=\big(L(a^+)L(a)\big)^*;$$
	$$R(a)R(a^+)R(a)=R(a), R(a^+)R(a)R(a^+)=R(a^+),R(a)R(a^+)=\big(R(a)R(a^+)\big)^*,R(a^+)R(a)=\big(R(a^+)R(a)\big)^*.$$
	\begin{equation}\label{qinvv}L(a)^+=L(a^+),\ R(b)^+=R(b^+),\ \big(L(a)R(b)\big)^+=L(a^+)R(b^+).\end{equation}
\end{pro}
Proposition \ref{mpprop} implies that the concept of Moore-Penrose inverse of split quaternions is well defined.

\begin{thm}Let $0\neq a=c_1+c_2\bj\in  Z(\bh_s)$ and $0\neq b=u_1+u_2\bj\in  Z(\bh_s)$.   Then the equation $axb=d$ is  solvable
	if and only if \begin{equation}\label{caxb}
	(1+\frac{c_2}{\bar{c_1}}\bj)d(1+\frac{u_2}{u_1}\bj)=4d,
	\end{equation} in which case all solutions are given by  \begin{equation}\label{saxb}
	x=\frac{\bar{c_1}d\bar{u_1}}{4|c_1|^2|u_1|^2}+y-\frac{1}{4}(1+\frac{c_2}{c_1}\bj)y(1+\frac{u_2}{\bar{u_1}}\bj),\forall y\in\bh_s.
	\end{equation}
\end{thm}

\begin{proof}
	It is obvious that $axb=d$ is equivalent to $L(a)R(b)\overrightarrow{x}=\overrightarrow{d}$.  By Lemma \ref{general} $axb=d$ is  solvable
	if and only if  $$L(a)R(b)\big(L(a)R(b)\big)^+\overrightarrow{d}=\overrightarrow{d}.$$  Returning to quaternion form  by Proposition \ref{mpprop}, we have $aa^+db^+b=d$. Using (\ref{cinvf}), we can rewrite  $aa^+db^+b=d$  as  (\ref{caxb}). By Lemma \ref{general},  the general solution is $$\overrightarrow{x}=\big((L(a)R(b)\big)^+\overrightarrow{d}+(E_4-\Big(\big(L(a)R(b)\big)^+L(a)R(b)\Big)\overrightarrow{y},\forall y\in \bh_s.$$  Hence the general solution can be expressed as$$x=a^+db^++(y-a^+aybb^+),\forall y\in\bh_s.$$
	That is \begin{equation*}	x=\frac{(\bar{c_1}+c_2\bj)d(\bar{u_1}+u_2\bj)}{16|c_1|^2|u_1|^2}+y-\frac{1}{4}(1+\frac{c_2}{c_1}\bj)y(1+\frac{u_2}{\bar{u_1}}\bj),\forall y\in\bh_s.
	\end{equation*}
	By  (\ref{caxb}) we have $\frac{(\bar{c_1}+c_2\bj)d(\bar{u_1}+u_2\bj)}{16|c_1|^2|u_1|^2}=\frac{\bar{c_1}d\bar{u_1}}{4|c_1|^2|u_1|^2}$.
\end{proof}

In similar way, we have the following corollaries.

\begin{cor}Let $0\neq a=c_1+c_2\bj\in  Z(\bh_s)$. Then the equation $ax=0$ has solutions: $$x=(1-a^+a)y=\frac{1}{2}(1-\frac{c_2}{c_1}\bj)y,\forall y\in \bh_s.$$
\end{cor}

\begin{cor}Let $0\neq a=c_1+c_2\bj\in  Z(\bh_s)$. Then the equation $ax=d$ is  solvable
	if and only if $aa^+d=\frac{1}{2}(1+\frac{c_2}{\bar{c_1}}\bj)d=d$, in which case all solutions are given by $$x=a^+d+(1-a^+a)y=\frac{\bar{c_1}+c_2\bj}{4|c_1|^2}d+\frac{1}{2}(1-\frac{c_2}{c_1}\bj)y,\forall y\in \bh_s.$$
\end{cor}

\begin{cor}\label{cor3} Let $0\neq a=c_1+c_2\bj\in  Z(\bh_s)$. Then the equation $xa=d$ is  solvable
	if and only if $da^+a=\frac{1}{2}d(1+\frac{c_2}{c_1}\bj)=d$, in which case all solutions are given by $$x=da^++y(1-aa^+)=d\frac{\bar{c_1}+c_2\bj}{4|c_1|^2}+\frac{y}{2}(1-\frac{c_2}{\bar{c_1}}\bj),\forall y\in \bh_s.$$
\end{cor}

\section{Similarity and consimilarity}\label{simsec}

\begin{defi}
	We say that two split quaternions $a,b\in \bh_s$  are similar (resp. consimilar)
	if and only if there exists a $q\in \bh_s-Z(\bh_s)$ such that $qa=bq$ (resp. $aq=\bar{q}b$).
\end{defi}
It is obvious that $xa=bx$ is equivalent to \begin{equation}(R(a)-L(b))\overrightarrow{x}=\overrightarrow{0}.\end{equation} Noting  the linear equations studied  in Section \ref{mpinvsec}, we  assume that $a,b\in \bh_s-\br$   for  $xa=bx$ in this section.

We can verify the following proposition directly.
 \begin{pro}\label{proped}
Let  $T(a,b)=R(a)-L(b)$. Then the eigenvalues of   $T(a,b)$ are
\begin{equation*}\lambda_{1,2,3,4}=a_0\pm \sqrt{K(a)}-\Big(b_0\pm \sqrt{K(b)}\Big)\end{equation*}
and
\begin{equation*}\det(T(a,b))=\Pi_i^4\lambda_i=(a_0-b_0)^4-2(a_0-b_0)^2\Big(K(a)+K(b)\Big)+\Big(K(a)-K(b)\Big)^2.\end{equation*}
$\det(T(a,b))=0$ if only if one of the following conditions holds:
\begin{itemize}
	\item [(1)]  $a_0=b_0$ and $K(a)=K(b)$; in this case $\mathrm{rank}(T(a,b))=2$.
\item [(2)]  $a_0-b_0\neq 0$ and $\det(T(a,b))=0$; in this case $K(a),K(b)\ge 0$ and  $\mathrm{rank}(T(a,b))=3$.
\end{itemize}
 \end{pro}

\begin{exam} Two examples of the $\det(T(a,b))=0$: (1)\quad
	$a=1+3\bi+2\bj+\bk,b=1+3\bi+\bj+2\bk$, $K(a)=K(b)=-4$, $\mathrm{rank}(T(a,b))=2$;  (2)\quad
	$a=2+\bi+\bk,b=1+\bk$, $K(a)=0, K(b)=1$, $\mathrm{rank}(T(a,b))=3$.
\end{exam}

\begin{lem}\label{sginv}
Let $a,b\in \bh_s-\br$ and $T=R(a)-L(b)$. If $a_0=b_0,K(a)=K(b)$ then $$T^+=\frac{R(a')-L(b')}{2(|\Im(a)|^2+|\Im(b)|^2)}.$$
\end{lem}
\begin{proof}
It follows from $a,b\in \bh_s-\br$ that $|\Im(a)|^2+|\Im(b)|^2\neq 0$.
Let $S=R(a')-L(b')$.  Since $a_0=b_0$, we have $S=R\big(\Im(a')\big)-L\big(\Im(b')\big)$  and  $T=R\big(\Im(a)\big)-L\big(\Im(b)\big)$.
Note that $$R\big(\Im(a)\Im(a')\big)+R\big(\Im(a')\Im(a)\big)=2|\Im(a)|^2E_4,L\big(\Im(b)\Im(b')\big)+L\big(\Im(b')\Im(b)\big)=2|\Im(b)|^2E_4,$$
$$R\big(\Im(a)\Im(a')\Im(a)\big)+L(\Im(b)^2)R(\Im(a'))=2|\Im(a)|^2R(\Im(a)),$$
$$-R\big(\Im(a)^2\big)L\big(\Im(b')\big)-L\big(\Im(b)\Im(b')\Im(b)\big)=-2|\Im(b)|^2L\big(\Im(a)\big).$$
Hence \begin{eqnarray*}TST&=&R\big(\Im(a)\Im(a')\Im(a)\big)-\big[R\big(\Im(a)\Im(a')\big)+R(\Im(a')\Im(a))\big]L\big(\Im(b)\big)+\big[L(\Im(b)\Im(b'))+L\big(\Im(b')\Im(b)\big)\big]R\big(\Im(a)\big)\\
&&-R\big(\Im(a)^2\big)L(\Im(b'))+L(\Im(b)^2)R(\Im(a'))-L\big(\Im(b)\Im(b')\Im(b)\big)=2(|\Im(a)|^2+|\Im(b)|^2)T.
\end{eqnarray*}
Similarly  we can verify that $STS=2(|\Im(a)|^2+|\Im(b)|^2)S$, $ST$ and $TS$ are symmetric matrices. 
\end{proof}

\begin{thm}\label{thmk0}
Let $a,b\in \bh_s-\br$ and $a_0=b_0,K(a)=K(b)$. Then the general solution of linear equation  $xa=bx$  is
\begin{equation}\label{thmk0e}x=y-\frac{yaa'-bya'-b'ya+b'by}{2(|\Im(a)|^2+|\Im(b)|^2)},\forall y\in \bh_s.\end{equation}
\end{thm}
\begin{proof}
It follows from   Lemma \ref{sginv} that the general solution of  linear equation  $xa=bx$  is $\overrightarrow{x}=(E_4-T^+T)\overrightarrow{y}$. Returning to quaternion form  by Proposition \ref{mpprop}, we get the formula (\ref{thmk0e}).
\end{proof}

\begin{thm}\label{a0nb0}
Let $a=c_1+c_2\bj,b=u_1+u_2\bj\in \bh_s-\br$, $a_0-b_0\neq 0$ and $\det(T(a,b))=0$. Then the general solution of linear equation  $xa=bx$  is
\begin{equation}\label{ran3s}
x=y\Big(1-\frac{2(a_0-b_0)c_2}{I_b-I_a+2(a_0-b_0)\bar{c_1}}\bj\Big)a-\bar{b}y\Big(1-\frac{2(a_0-b_0)c_2}{I_b-I_a+2(a_0-b_0)\bar{c_1}}\bj\Big),\forall  y\in \bh_s.
\end{equation}
\end{thm}
\begin{proof}
It follows from   Proposition \ref{proped} that $\det\big(R(a)-L(b)\big)=\det\big(R(a)-L(\bar{b})\big).$
Let $$p= I_b-I_a+2(a_0-b_0)a=I_b-I_a+2(a_0-b_0)c_1+2(a_0-b_0)c_2\bj.$$  Then $I_p=\det(T(a,b))=0,$  which implies that $0\neq p\in Z(\bh_s)$.  Note that \begin{equation}\big(R(a)-L(b)\big)\big(R(a)-L(\bar{b})\big)=R(p).\end{equation} 
Let $x_1$  be the general solution of the equation  $xp=0$. Then $\big(R(a)-L(b)\big)\big(R(a)-L(\bar{b})\big)\overrightarrow{x_1}=\overrightarrow{0}$. By Corollary \ref{cor3}, $$x_1=y\Big(1-\frac{2(a_0-b_0)c_2}{I_b-I_a+2(a_0-b_0)\bar{c_1}}\bj\Big).$$  It follows from $\mathrm{rank}(T(a,b))=3$ that  the general solution of $xa=bx$ is $x=x_1a-\bar{b}x_1$, which is just  (\ref{ran3s}).
\end{proof}

\begin{rem}
	Kula and Yayli \cite[Theorem 5.5]{yayli} didn't treat the case $a_0=b_0,K(a)=K(b)=0$; for example $a=1+5\bi+3\bj+4\bk,b=1+13\bi+12\bj+5\bk$. Also they  overlooked  the case of Theorem \ref{a0nb0}, therefore \cite[Proposition 5.3 (ii)]{yayli} should be amended.  For example, $a=1+5\bi+5\bj+2\bk,b=2+\bi+\bj+3\bk$; in this case, the eigenvalues of $R(a)-L(b)$ are $-6,-2,0,4$ and $rank(R(a)-L(b))=3$. The solutions of  $xa=bx$ can be represented by $$x=y(1-\bk)a-\bar{b}y(1-\bk)=r(-3+\bi+\bj+3\bk),\forall r\in \br.$$
We mention that  taking $y=1$ leads to  $r=2$ in the above formula.
\end{rem}


\begin{lem}\label{similar}
Let  $a=a_0+a_1\bi+a_2\bj+a_3\bk\in \bh_s-\br$. Then there exists a  $q\in \bh_s-Z(\bh_s)$ such that $$qaq^{-1}=\left\{                                                                                                                             \begin{array}{lll}                                                                                                                                    a_0+\sqrt{K(a)}\ \bj, & \hbox{provided $K(a)>0$;} \\                                                                                                                         a_0+\sqrt{-K(a)}\ \bi, & \hbox{provided $K(a)<0$;}\\                                                                                                                                  a_0+\bi+\bj,& \hbox{provided $K(a)=0$.}
\end{array}                                                                                                                             \right.
$$
\end{lem}
\begin{proof}
	Kula and Yayli \cite[Propositions 5.1 and 5.2]{yayli} treated the case $K(a)\neq 0$. We only need to consider the case $K(a)=0$. We at first show that we can conjugate $a$ to $a_0+a_1\bi-a_1\bj$.

If $a_3=0$ then  $a=a_0+a_1\bi+a_2\bj$ and $|a_1|=|a_2|$. Note that \begin{equation}\label{ijk}\bi(a_1\bi+a_2\bj)\bi^{-1}=a_1\bi-a_2\bj,\  \bj(a_1\bi+a_2\bj)\bj^{-1}=-a_1\bi+a_2\bj,\ \bk(a_1\bi+a_2\bj)\bk^{-1}=-a_1\bi-a_2\bj.\end{equation}
Hence we can find a $p\in \bh_s-Z(\bh_s)$ such that $pap^{-1}=a_0+a_1\bi-a_1\bj$.

If $a_3\neq 0$ then $a_1^2=a_2^2+a_3^2>a_2^2$. In Theorem \ref{thmk0}, letting $b=a_0+a_1\bi-a_1\bj$ and taking $y=2a_1(1+\bk)$ in (\ref{thmk0e}), we get $p=a_1-a_2+a_3\bi,I_p=2a_1(a_1-a_2)\neq 0.$
That is  $pap^{-1}=a_0+a_1\bi-a_1\bj$.

If $a_1=1$ or $a_1=-1$, then by (\ref{ijk}), we can find a   $q$ such that $qaq^{-1}=a_0+\bi+\bj$. Otherwise, letting $a=a_0+a_1\bi-a_1\bj$ and $b= a_0+\bi+\bj$ and taking $y=\frac{2(a_1^2+1)\bi}{a_1+1}$ in (\ref{thmk0e}), we get  $q=(1+a_1)\bi+(1-a_1)\bj$ such that $q(a_0+a_1\bi-a_1\bj)q^{-1}=a_0+\bi+\bj$.
\end{proof}

The fact that a real number is similar to itself, together with Lemma \ref{similar}, implies the following theorem.

\begin{thm}\label{thmsim}
	Two split quaternions $a,b \in \bh_s-\br$ are similar if and only if
	$$\Re(a)=\Re(a), K(a)=K(a).$$
\end{thm}

Note that $xa=b\bar{x}$ is equivalent to $\big(R(a)-L(b)F\big)\overrightarrow{x}=0.$  We can verify the following proposition directly.
\begin{pro}\label{propsab}
  Let $F=diag(1,-1,-1,-1)$ and $S(a,b)=R(a)-L(b)F$.	
Then the eigenvalues of $S(a,b)$ are
$$\lambda_{1,2}=a_0 \pm \big(K(a)-K(b)+b_0^2\big)^\frac{1}{2}=a_0 \pm \big(K(a)+I_b\big)^\frac{1}{2},$$
$$\lambda_{3,4}=a_0+b_0\pm \big(K(a)+K(b)+2(a_1b_1-a_2b_2-a_3b_3)\big)^\frac{1}{2}$$
and
$$\det(S(a,b))=\Pi_i^4\lambda_i=(I_a-I_b)\big(I_a+I_b+2(a_0b_0-a_1b_1+a_2b_2+a_3b_3)\big)=(I_a-I_b)I_{\bar{a}+b}.$$	
$\det(S(a,b))=0$ if and only if one of the following conditions holds:
\begin{itemize}
	\item [(1)] $I_a=I_b$ and $\bar{a}+b=0$; in this case $\mathrm{rank}(S(a,b))=1$.
\item [(2)]  $I_a=I_b$ and $I_{\bar{a}+b}\neq 0$; in this case   $\mathrm{rank}(S(a,b))=3$.
\item [(3)]  $I_a=I_b$ and $0\neq \bar{a}+b\in Z(\bh_s)$; in this case   $\mathrm{rank}(S(a,b))=3$.
\item [(4)] $I_a\neq I_b$ and $0\neq \bar{a}+b\in Z(\bh_s)$;  in this case   $\mathrm{rank}(S(a,b))=3$.
\end{itemize}
 \end{pro}

\begin{exam} Four examples of  $\det(S(a,b))=0$: (1)\quad
$a=1+2\bi+3\bj+4\bk, b=-1+2\bi+3\bj+4\bk$, $\bar{a}+b=0$,  $\mathrm{rank}(S(a,b))=1$; (2)\quad $a=1+2\bi+3\bj+4\bk,b=2+\bi+3\bj+4\bk$, $I_{\bar{a}+b}=10$, $I_a=I_b=-20, \mathrm{rank}(S(a,b))=3$; (3)\quad $a=1+2\bi+3\bj+4\bk, b=-2+\bi+4\bj+3\bk$, $I_{\bar{a}+b}=0$, $I_a=I_b=-20, \mathrm{rank}(S(a,b))=3$;
(4)\quad $a=1+2\bi+3\bj+4\bk,b=2+\bi+3\bk$, $I_a=-20,I_b=-4,I_a\neq I_b$, $I_{\bar{a}+b}=0, \mathrm{rank}(S(a,b))=3$.
\end{exam}

\begin{thm}\label{thmsim}
	Two split quaternions $a,b \in \bh_s-\br$ are consimilar if and only if one of the following two conditions holds:
	$$(1)\quad   \bar{a}+b=0;\  \quad  (2)\quad  I_a=I_b, I_{\bar{a}+b}\neq 0.$$
\end{thm}

\begin{proof}
	If $a,b$ are consimilar then there exists an $x\in  \bh_s-Z(\bh_s)$ such that $xa=b\bar{x}$. Thus $xa\overline{xa}=b\bar{x}x\bar{b}$. That is $I_x(I_a-I_b)=0$, which implies that $I_a=I_b$.
	 It is obvious that $x=\bar{a}+b$ is a solution to $xa=b\bar{x}$.  If $\bar{a}+b=0$ then the following three split quaternions $x=a_3\bi+a_1\bk$, $x=a_2\bi+a_1\bj$ and $x=a_1+a_0\bi$ are solutions to $xa=b\bar{x}$, among them there  is at least one number in  $\bh_s-Z(\bh_s)$. If $x=\bar{a}+b\neq 0$ then it is a basis of the solution space of $xa=b\bar{x}$.   The above observation and Proposition \ref{propsab} imply Theorem  \ref{thmsim}.
\end{proof}

{\bf Acknowledgements.}\quad This work is supported by Natural Science Foundation of China (no:11871379 ) and the Innovation Project of Department of Education of Guangdong Province.

\end{document}